\documentclass[leqno,a4paper]{article}

\usepackage{amsthm,amsfonts,amssymb,amsmath,amsxtra}
\usepackage[all]{xy}
\usepackage{xr-hyper}
\usepackage[colorlinks=true]{hyperref}
\usepackage{makeidx}\makeindex
\usepackage{mathrsfs}

\theoremstyle{plain}
\newtheorem{thm}{Theorem}[section]
\newtheorem{lem}[thm]{Lemma}

\newtheorem{prop}[thm]{Proposition}
\theoremstyle{definition}

\newtheorem{remark}[thm]{Remark}

\numberwithin{equation}{section}

    \newcommand{\sS}{{\mathscr{S}}}

    \newcommand{\BC}{{\mathbb {C}}}

    \newcommand{\BQ}{{\mathbb {Q}}}

    \newcommand{\BW}{{\mathbb {W}}} \newcommand{\BX}{{\mathbb {X}}}
    \newcommand{\BY}{{\mathbb {Y}}}

     \newcommand{\CR}{{\mathcal {R}}}

    \newcommand{\GU}{{\mathrm{GU}}}

    \newcommand{\GSp}{{\mathrm{GSp}}}
    \newcommand{\GMp}{{\mathrm{GMp}}}
     
    \newcommand{\GL}{{\mathrm{GL}}}
    \newcommand{\Hom}{{\mathrm{Hom}}}
    
    \newcommand{\id}{{\mathrm{id}}}
     \newcommand{\ind}{{\mathrm{ind}}}
    \newcommand{\Irr}{{\mathrm{Irr}}}

    \newcommand{\Ker}{{\mathrm{Ker}}}

    \newcommand{\Mp}{{\mathrm{Mp}}}

     \newcommand{\N}{{\mathrm{N}}}

    \newcommand{\pr}{{\mathrm{pr}}}

    \newcommand{\Sp}{{\mathrm{Sp}}}

    \newcommand{\RTr}{{\mathrm{Tr}}}

    \newcommand{\U}{{\mathrm{U}}}

    \newcommand{\wt}{\widetilde}

    \newcommand{\pair}[1]{\langle {#1} \rangle}

    \newcommand{\sk}{\medskip}
    \newcommand{\lra}{\longrightarrow}
    \newcommand{\ra}{\rightarrow}

    \newcommand{\s}{\sk\noindent}

    \numberwithin{equation}{section}

\begin{document}

\title{A note on the local theta correspondence for unitary similitude dual pairs}

\author{Chong Zhang}

\maketitle

\begin{abstract}
Following Roberts' work in the case of orthogonal-symplectic
similitude dual pairs, we study the local theta correspondence for
unitary similitude dual pairs over a $p$-adic field.
\end{abstract}

\section{Introduction}\label{sec: intro}

The local theta correspondence has been well studied for isometry
groups. Meanwhile the theta correspondence for similitude groups is
also important and useful. Roberts \cite{ro} (and others) studied
the local theta correspondence (over a non-archimedean local field)
in the orthogonal-symplectic case. In this short note, under the
assumption ({\bf split}): \emph{one of the hermitian and
skew-hermitian spaces is split}, we show that the method of
\cite{ro} is also valid for the unitary case. We remark that part of
our results may be known to experts, for example, see some works of
M. Harris (e.g. \cite{ha1}). Since we could not find a literature
for the statements and proofs of these results, we insist to write
this note.

After this note was written, Gan informed the author that
\cite[Proposition 3.2]{gt} could show that Howe duality conjecture
for isometry case implies Howe duality conjecture for similitude
case for general dual pairs and the argument of
multiplicity-freeness in \cite{ro} is not needed. Though only the
case of quaternionic unitary groups is treated in \cite[Proposition
3.2]{gt}, it is clear that the argument works generally for other
dual pairs.

In Section \ref{sec: notation}, we recall some notations and
definitions that will be used later. In Section \ref{sec:
splitting}, we study the problem of splitting metaplectic covers of
unitary similitude dual pairs. We show in Proposition \ref{prop
splitting} that this is not always available contrary to the
isometry case. We also show that the preimage of the similitude dual
pairs in the metaplectic cover are not always commute (cf.
Proposition \ref{prop commute}). Actually, the assumption ({\bf
split}) is only essentially used in this section. In Section
\ref{sec: Howe duality}, we study the Howe duality for the
similitude groups. Our main result is Theorem \ref{thm Howe}.

\section{Notations and assumptions}\label{sec: notation}
Let $F$ be a $p$-adic field (i.e. a finite extension field of
$\BQ_p$), and $E$ be a quadratic extension field of $F$.
We will always assume $p\neq2$ in \S4. Let
$\epsilon_{E/F}$ be the quadratic character of $F^\times$ associated
to $E$ by the local class field theory. Denote by $x\mapsto\bar{x}$
the non-trivial Galois automorphism of $E/F$. Choose an
element $\delta\in E^\times$ such that $\bar{\delta}=-\delta$. Then
$\Delta:=\delta^2\in F^\times$ and $\epsilon_{E/F}(x)=(x,\Delta)_F$,
where $(\ ,\ )_F$ is the Hilbert symbol for $F$. We write $\N_{E/F}$
for the norm map from $E^\times$ to $F^\times$. Fix a non-trivial
additive character $\psi:F\ra\BC^\times$, and denote
$\eta=\frac{1}{2}\psi$.

Let $(W,\pair{\ ,\ })$ (resp. $(V,(\ ,\ ))$) be a non-degenerate
skew hermitian (resp. hermitian) space over $E$ with $\dim_EW=n$
(resp. $\dim_EV=m$). Let $H=\GU(W)$ and $G=\GU(V)$ be the
corresponding unitary similitude groups. Recall that $h\in H$
if and only if $h\in\GL_E(W)$ and there exists $\nu\in F^\times$
such that $\pair{xh,yh}=\nu\pair{x,y}$ for all $x,y\in W$.
If $h\in H$, such $\nu$ is unique and will be denoted by $\nu(h)$.
Thus we have the scale map $\nu:H\ra F^\times$. The group $G=\GU(V)$
is defined similarly, and we still use $\nu$ to denote the scale map
$G\ra F^\times$.
Denote $H_1=\U(W)$ and $G_1=\U(V)$, which are the kernels of $\nu$.

Since the distinguish between hermitian space and skew-hermitian
space is not essential, we will assume that ({\bf split}): {\em $W$
is split}. In other words, there is a complete polarization $W=X+Y$,
where $X$ and $Y$ are maximal isotropic subspaces of $W$.

Attached to $(V,W)$, there is a symplectic space $\BW=V\otimes_EW$
over $F$, equipped with the symplectic form $$\pair{\pair{\ ,\
}}=\frac{1}{2}\RTr_{E/F}\left((\ ,\ )\otimes\overline{\pair{\ ,\
}}\right).$$ There is a natural embedding $$\iota:G\times
H\lra\GSp(\BW),\ (v\otimes w)\iota(g,h)=(g^{-1}v\otimes wh),$$ where
$\GSp(\BW)$ is the symplectic similitude group attached to $\BW$
(note that we write the action of $\GSp(\BW)$ or $\Sp(\BW)$ on
right). We denote by $\iota_V:H\ra\GSp(\BW)$ and
$\iota_W:G\ra\GSp(\BW)$ the restrictions of $\iota$ on $H$ and $G$.

Let $\Mp(\BW)$ be the metaplectic cover of $\Sp(\BW)$, which is a
non-trivial $\BC^1$-central extension of $\Sp(\BW)$. For the fixed
additive character $\psi$, let $\omega_\psi$ be the smooth Weil
representation of $\Mp(\BW)$.

For convenience, we recall some notations used for the (strong) Howe
duality (cf. \cite{ro}). Let $A$ and $B$ be two groups of td-type
with countable bases, and $(\rho,S)$ be a smooth representation of
$A\times B$. For a group $J$ of td-type with countable bases, denote
by $\Irr(J)$ the set of equivalence classes of irreducible
admissible representations of $J$. For $\pi\in\Irr(A)$, put
$S(\pi)=S/\bigcap_{f\in\Hom_A(\rho,\pi)}\Ker(f)$. It is a smooth
representation of $A\times B$, and there exists a smooth
representation $\Theta(\pi)$ of $B$, unique up to isomorphism, such
that $S(\pi)\simeq\pi\otimes\Theta(\pi)$ (cf. \cite{mvw}). Denote by
$\CR(A)$ the set of equivalence classes of $\pi\in\Irr(A)$ such that
$S(\pi)\neq0$, and denote $\CR(B)$ analogously. We say that {\em
strong Howe duality} holds for $\rho$ if for every $\pi\in\CR(A)$
the representation $\Theta(\pi)$ has a unique nonzero irreducible
quotient $\theta(\pi)$, and analogous statement holds for every
$\tau\in\CR(B)$. We say that {\em Howe duality} holds for $\rho$ if
the set $\{(\pi,\tau)\in\CR(A)\times\CR(B):\ \Hom_{A\times
B}(\rho,\pi\otimes\tau)\neq0\}$ is the graph of a bijection between
$\CR(A)$ and $\CR(B)$. It is known that strong Howe duality implies
Howe duality.

\section{The splitting}\label{sec: splitting}
Notice that we assume $W$ is split with $\dim_EW=n=2r$, that is
there is a complete polarization $W=X+Y$. Let $\BX=V\otimes_EX$ and
$\BY=V\otimes_EY$. Then $\BX+\BY$ is a complete polarization of
$\BW$. If necessary we write the elements of $H$ or $\GSp(\BW)$ as
matrices with respect to these polarizations.

Let $\GMp(\BW)$ be the metaplectic cover of $\GSp(\BW)$ (cf.
\cite{ba}), which is an extension of $\GSp(\BW)$ by $\BC^1$:
$$1\lra\BC^1\lra\GMp(\BW)\stackrel{\pr}{\lra}\GSp(\BW)\lra1.$$
Then $\GMp(\BW)\simeq\GSp(\BW)\times\BC^1$ as sets, and the group
law can be written as
$$(g,z)\cdot(g',z')=(gg',C(g,g')zz'),$$ where $g,g'\in\GSp(\BW),\
z,z'\in\BC^1$, and $C:\GSp(\BW)\times\GSp(\BW)\ra\BC^1$ is the
cocycle function. The cocycle $C$ can be computed explicitly as
follows (cf. \cite[\S1.2]{ba}). There is an action of $F^\times$ on
$\Sp(\BW)$ defined by $s^y=d(y)^{-1}sd(y)$, where $s\in\Sp(\BW),\
y\in F^\times$ and $d(y)=\begin{pmatrix}\mathbf{1}_{mn}&0\\0&y\cdot
\mathbf{1}_{mn}\end{pmatrix}$. For $g\in\GSp(\BW)$, denote
$g_1=d(\nu(g))^{-1}g\in\Sp(\BW)$. Then we have
\begin{equation}\label{equation cocycle C}C(g,g')=
c_\BY(g_1^{\nu(g')},g'_1)\mu(\nu(g'),g_1),\end{equation} where
$c_\BY(g_1,g'_1)=\gamma_F(\eta\circ L(\BY,\BY {g'_1}^{-1},\BY
g_1))$, and $\mu: F^\times\times\Sp(\BW)\ra\BC^1$ is the functions
defined by
\begin{equation}\label{equation function mu}\mu(y,s)=
(x(s),y)_F\gamma_F(y,\eta)^{j(s)}.\end{equation} For convenience we
recall some of the notations appeared above. For the triple of
Lagrangians $(\BY,\BY {g'_1}^{-1},\BY g_1)$, $L(\BY,\BY
{g'_1}^{-1},\BY g_1)$ is its Leray invariant, which is a quadratic
space over $F$ (cf. \cite{ra} for the definitions).
For a quadratic space $q$ over $F$ and an additive
character $\eta$, $\gamma_F(\eta\circ q)$ is the Weil index
associated to $q$ and $\eta$. For $y\in F^\times$,
$\gamma_F(y,\eta)$ is defined to be $\frac{\gamma_F(y\eta\circ
x^2)}{\gamma_F(\eta\circ x^2)}$, where $y\eta(x):=\eta(yx)$.
Let $\{e_1,...,e_{mn},e_1^*,...,e_{mn}^*\}$ be a basis of $\BW$ such that
$\{e_1,...,e_{mn}\}\subset\BX$,
$\{e_1^*,...,e_{mn}^*\}\subset\BY$,
$\pair{e_i,e_j}=\pair{e^*_i,e^*_j}=0$ and
$\pair{e_i,e_j^*}=\delta_{ij}$ for $1\leq i,j\leq mn$.
Let $P_\BY$ be the parabolic
subgroup of $\Sp(\BW)$ stabilizing $\BY$. Then
$$\Sp(\BW)=\bigsqcup_{j=0}^{mn} P_\BY\tau_j P_\BY,$$ where $\tau_j$ is
defined to be \begin{equation}\label{equation weyl
elements}\begin{aligned} &e_i\cdot\tau_j=-e_i^*,\ &\textrm{and}\quad
e_i^*\cdot\tau_j=e_i,\quad
&\textrm{if }i\in\{1,...,j\},\\
&e_i\cdot\tau_j=e_i,\ &\textrm{and}\quad
e_i^*\cdot\tau_j=e_i^*,\quad &\textrm{otherwise},
\end{aligned}\end{equation}
and $\tau_0:=\id$. Thus for $s=p_1\tau_j p_2\in\Sp(\BW)$ with
$p_1,p_2\in P_\BY$, define $x(s):=\det(p_1p_2|_Y)$ in
$F^\times/F^{\times2}$ and $j(s):=j$.

Now we study the preimages $\pr^{-1}(G)$ and $\pr^{-1}(H)$ in
$\GMp(\BW)$. For this, we need some preparations. Let
$\{e_1,...,e_r,e_1^*,...,e_r^*\}$ be a basis of $W$ such that
$\{e_1,...,e_r\}\subset X$, $\{e_1^*,...,e_r^*\}\subset Y$,
$\pair{e_i,e_j}=\pair{e^*_i,e^*_j}=0$ and
$\pair{e_i,e_j^*}=\delta_{ij}$ for $1\leq i,j\leq r$. Let $P_Y$ be
the parabolic subgroup of $H_1$ stabilizing $Y$. Then we have Bruhat
decomposition $H_1=\bigsqcup_{i=0}^r P_Y\tau_j P_Y$, where $\tau_j$
is defined in the same way as (\ref{equation weyl elements}). For
$h=p_1\tau_jp_2\in H_1$ with $p_1,p_2\in P_Y$, define
$x(h):=\det(p_1p_2|_Y)$ in $E^\times/\N_{E/F}(E^\times)$ and
$j(h):=j$. Recall that, fixing a character
$\chi:E^\times\ra\BC^\times$ such that
$\chi|_{F^\times}=\epsilon^m_{E/F}$, there is a splitting
homomorphism (cf. \cite[\S3]{ku}) $$\wt{\iota}_{V,\chi}:
H_1\lra\Mp(\BW)\simeq\Sp(\BW)\times\BC^1,\quad h\mapsto
\wt{\iota}_{V,\chi}(h)=(\iota_V(h),\beta_{V,\chi}(h)),$$ where
$\beta_{V,\chi}(h)=\chi(x(h))\gamma_F(\eta\circ RV)^{-j(h)}$ and
$RV$ is the underlying $2m$ dimensional $F$ vector space with
quadratic form $\frac{1}{2}\RTr_{E/F}(\ ,\ )$. Recall that we have
the relation:
\begin{equation}\label{equation relation}
c_\BY(\iota_V(h),\iota_V(h'))=\beta_{V,\chi}(h)^{-1}\beta_{V,\chi}(h')^{-1}
\beta_{V,\chi}(hh'),\quad h,h'\in H_1.
\end{equation}
Meanwhile, there is a natural splitting
$$\wt{\iota}_W: G_1\lra \Mp(\BW)\simeq\Sp(\BW)\times\BC^1,
\quad g\mapsto(\iota_W(g),1).$$ Similarly as $\GSp(\BW)$, for $y\in
F^\times$ and $h\in H_1$, define
$$d(y)=\begin{pmatrix}\mathbf{1}_r&0\\0&y\cdot
\mathbf{1}_r\end{pmatrix},\quad\textrm{and}\quad
h^y=d(y)^{-1}hd(y)\in H_1.$$ For $h\in H$, denote
$h_1=d(\nu(h))^{-1}h\in H_1$. The following lemma will be used in
the proof of Proposition \ref{prop splitting}. It can be checked by
linear algebra, and we omit the proof.

\begin{lem}\label{lem linear algebra}
For $h\in H_1$ and $y\in F^\times$, we have
\begin{enumerate}
\item $x(h^y)=x(h)y^{j(h)}$ and $j(h^y)=j(h)$;
\item $x(\iota_V(h))=\N_{E/F}(x(h))^m\cdot(-\Delta)^{mj(h)}$
and $j(\iota_V(h))=2mj(h)$.
\end{enumerate}
\end{lem}

\begin{prop}\label{prop splitting}
As an extension of $H$, $\pr^{-1}(H)$ is trivial if $m$ is even. As
an extension of $G$, $\pr^{-1}(G)$ is trivial.
\end{prop}

\begin{proof}
For $h\in H$, it is easy to see that $\iota_V(h)_1=\iota_V(h_1)$ and
$\nu(h)=\nu(\iota_V(h))$. For $h,h'\in H$, by (\ref{equation cocycle
C}) and (\ref{equation relation}), we have
$$
\begin{aligned}C(\iota_V(h),\iota_V(h'))
&=c_{\BY}(\iota_V(h_1)^{\nu(h')},\iota_V(h'_1))\cdot\mu(\nu(h'),\iota_V(h_1))\\
&=c_\BY(\iota_V(h_1^{\nu(h')}),\iota_V(h'_1))\cdot\mu(\nu(h'),\iota_V(h_1))\\
&=\beta_{V,\chi}(h_1^{\nu(h')})^{-1}\beta_{V,\chi}(h'_1)^{-1}
\beta_{V,\chi}(h_1^{\nu(h')}h'_1)\cdot\mu(\nu(h'),\iota_V(h_1)).
\end{aligned}
$$
Since $$h_1^{\nu(h')}h'_1=d(\nu(h'))^{-1}d(\nu(h))^{-1}h
d(\nu(h'))\cdot d(\nu(h'))^{-1}h'=(hh')_1,$$ we have
$\beta_{V,\chi}(h_1^{\nu(h')}h'_1)=\beta_{V,\chi}((hh')_1)$. On the
other hand, by (\ref{equation function mu}) and Lemma \ref{lem
linear algebra},we have
$$\begin{aligned}\mu(\nu(h'),\iota_V(h_1))
&=(x(\iota_V(h_1)),\nu(h'))_F\gamma_F(\nu(h'),\eta)^{j(\iota_V(h_1))}\\
&=(\N_{E/F}(x(h_1)),\nu(h'))^m_F(-\Delta,\nu(h'))_F^{mj(h_1)}
\gamma_F(\nu(h'),\eta)^{2mj(h_1)}\\
&=(\N_{E/F}(x(h_1)),\nu(h'))^m_F(-\Delta,\nu(h'))_F^{mj(h_1)}
(-1,\nu(h'))_F^{mj(h_1)}\\
&=(\N_{E/F}(x(h_1)),\nu(h'))^m_F(\Delta,\nu(h'))_F^{mj(h_1)},\end{aligned}$$
and
$$\begin{aligned}\beta_{V,\chi}(h_1^{\nu(h')})
&=\chi(x(h_1^{\nu(h')}))\gamma_F(\eta\circ RV)
^{-j(h_1^{\nu(h')})}\\
&=\chi(x(h_1)\nu(h')^{j(h_1)})\gamma_F(\eta\circ
RV)^{-j(h_1)}\\
&=\chi(x(h_1))\gamma_F(\eta\circ
RV)^{-j(h_1)}\cdot\epsilon_{E/F}^{mj(h_1)}(\nu(h'))\\
&=\beta_{V,\chi}(h_1)(\nu(h'),\Delta)_F^{mj(h_1)}.\end{aligned}$$
Thus we have
$$\beta_{V,\chi}(h_1^{\nu(h')})^{-1}\mu(\nu(h'),\iota_V(h_1))
=\beta_{V,\chi}(h_1)^{-1}(\N_{E/F}(x(h_1)),\nu(h'))^m_F.$$ It
follows that
$$C(\iota_V(h),\iota_V(h'))=\beta_{V,\chi}(h_1)^{-1}
\beta_{V,\chi}(h'_1)^{-1}\beta_{V,\chi}((hh')_1)
\cdot(\N_{E/F}(x(h_1)),\nu(h'))^m_F.$$ Therefore $\pr^{-1}(H)$ is a
trivial $\BC^1$-extension if $m$ is even.

For $g\in G$, it is easy to see that $\nu(\iota_W(g))=\nu(g)^{-1}$
and $x(\iota_W(g)_1)=\N_{E/F}(\det g)^r$. For $g,g'\in G$, we have
$$\begin{aligned}
C(\iota_W(g),\iota_W(g'))&=c_\BY(\iota_W(g)_1^{\nu(g')^{-1}},\iota_W(g')_1)\cdot
\mu(\nu(g')^{-1},\iota_W(g)_1)\\
&=\mu(\nu(g')^{-1},\iota_W(g)_1)=(\nu(g')^{-1},\N_{E/F}(\det
g)^r)_F\\
&=(\nu(g),\nu(g'))^{mr}_F,
\end{aligned}$$
since $\iota_W(g)_1\in P_\BY$ and $\N_{E/F}(\det g)=\nu(g)^m$. Then
the conclusion follows from the relation
$$(\nu(g),\nu(g'))_F=\gamma_F(\nu(g),\eta)^{-1}\gamma_F(\nu(g'),\eta)^{-1}
\gamma_F(\nu(gg'),\eta).$$
\end{proof}

\begin{prop}\label{prop commute}
$\pr^{-1}(H)$ and $\pr^{-1}(G)$ commute in $\GMp(\BW)$ if and only
if $m$ is even.
\end{prop}

\begin{proof}
For $g\in G$ and $h\in H$, if $\pr(\wt{g})=g$ and $\pr(\wt{h})=h$,
we can write $\wt{g}=(\iota_W(g),z)$ and $\wt{h}=(\iota_V(h),z')$.
It is easy to see that the commutator of $\wt{g}$ and $\wt{h}$ is
$$[\wt{g},\wt{h}]=\frac{C(\iota_W(g),\iota_V(h))}
{C(\iota_V(h),\iota_W(g))}.$$ Then we have
$$\begin{aligned}C(\iota_W(g),\iota_V(h))
&=c_\BY(\iota_W(g)_1^{\nu(h)},\iota_V(h_1))
\mu(\nu(h),\iota_W(g)_1)\\
&=(\N_{E/F}(\det(g)),\nu(h))^r_F\\
&=(\nu(g),\nu(h))^{mr}_F,\end{aligned}$$ and
$$\begin{aligned}C(\iota_V(h),\iota_W(g))
&=c_\BY(\iota_V(h_1)^{\nu(g)^{-1}},\iota_W(g)_1)
\mu(\nu(g)^{-1},\iota_V(h_1))\\
&=\mu(\nu(g)^{-1},\iota_V(h_1))\\
&=(\N_{E/F}(x(h_1)),\nu(g)^{-1})^m_F(\Delta,\nu(g)^{-1})_F^{mj(h_1)}.\end{aligned}$$
Hence $\pr^{-1}(H)$ and $\pr^{-1}(G)$ commute in $\GMp(\BW)$ if and
only if $m$ is even.
\end{proof}

\section{Howe duality}\label{sec: Howe duality}
Let $\omega_\psi$ be the smooth Weil representation of $\Mp(\BW)$.
Using the fixed splitting homomorphisms as in the previous section,
we have the Weil representation $\omega_{\psi,\chi}$ of $G_1\times
H_1$. We simply denote by $\omega$ for $\omega_{\psi}$ or
$\omega_{\psi,\chi}$. Let $\Omega=\ind_{\Mp(\BW)}^{\GMp(\BW)}\omega$
be the compactly induced representation, and recall that there is an
one extra variable Schr$\mathrm{\ddot{o}}$dinger model
$\sS(\BX\times F^\times)$ of $\Omega$ (cf. \cite[\S1.3]{ba}). If
$m=\dim_EV$ is even, by Proposition \ref{prop splitting} and
Proposition \ref{prop commute}, $\Omega$ can be viewed as a
representation of $G\times H$. However, when $m$ is odd, the above
method of constructing representation is not available. For general
$m$, we have the following construction (cf. \cite[\S3]{ro}).

Let $$R=\{\ (g,h)\in G\times H\ |\ \nu(g)=\nu(h)\}$$ be a subgroup
of $G\times H$ and also a subgroup of $\Sp(\BW)$. We can extend the
Weil representation $\omega|_{G_1\times H_1}$ to $R$, and still
denote it by $\omega$. Let
$$H^+=\{\ h\in H\ |\ \nu(h)\in\nu(G)\}.$$ Then $H^+=H$ when $m$ is even;
$[H:H^+]=2$ when $m$ is odd. Let
$$\Omega^+=\ind_{R}^{G\times H^+}\omega$$ be the compactly induced
representation, which is a smooth representation of $G\times H^+$.

\begin{remark} If $m$ is even, we thus have two representations $\Omega^+$
and $\Omega$ of $G\times H$. However it can be shown that
$\Omega^+\simeq\Omega$, by the same arguments of \cite[Proposition
3.5]{ro}.
\end{remark}

\begin{remark}
If $W$ (resp. $V$) is an arbitrary skew-hermitian (resp. hermitian)
space, we can define $H^+$ with respect to $V$ (resp. $G^+$ with
respect to $W$), the subgroup $R$ of $G^+\times H^+$, and the
representation $\Omega^+=\ind_R^{G^+\times H^+}\omega$ in the same
way.
\end{remark}

By such a construction, we can consider the (strong) Howe duality
for the representation $\Omega^+$ of $G\times H^+$.

\begin{thm}\label{thm Howe}
If $p>2$, the strong Howe duality holds for $\Omega^+$.
\end{thm}

\begin{proof}
If $p>2$, Waldspurger \cite{wa} proved that the strong Howe duality
holds for $\omega|_{G_1\times H_1}$. The centers of $G,\ H^+$ are
both isomorphic to $E^\times$, and there are isomorphisms:
$$G/E^\times G_1\simeq R/E^\times(G_1\times H_1)\simeq H^+/E^\times
H_1,$$ where we use $E^\times$ to denote the centers of $G$ and
$H^+$, and embed $E^\times$ into $R$ diagonally (then $E^\times$
lies in the center of $R$). Notice that the cardinality of the above
quotient groups is no more than 2. Thus, for $\pi\in\Irr(H^+)$, we
have either $\pi|_{H_1}\simeq\rho_0$ being irreducible or
$\pi|_{H_1}\simeq\rho_1\bigoplus\rho_2$ where $\rho_i\in\Irr(H_1)$
are pairwise inequivalent. Then it can be checked that the arguments
in Section 4 and 5 of \cite{ro} can also be applied to our
situation, and the conclusion follows. More precisely, for
$\pi\in\Irr(H^+)$, if $\pi|_{H_1}\simeq\bigoplus\rho_i$ so that
$\rho_i\in\Irr(H_1)$, then $\theta(\pi)=\bigoplus\theta(\rho_i)$
where $\theta(\rho_i)\in\Irr(G_1)$.
\end{proof}

\begin{remark}
If $W$ is arbitrary (i.e. not necessary split), applying the above
arguments again and noting the key relation $G^+/E^\times G_1\simeq
R/E^\times(G_1\times H_1)\simeq H^+/E^\times H_1$, we can show that
if $p>2$ then the (strong) Howe duality holds for $\Omega^+$.
\end{remark}

As in \cite[\S6]{ro}, when $m$ is odd, we can also consider the
(strong) Howe duality for the compactly induced representation
$$\wt{\Omega}=\ind_R^{G\times H}\omega.$$ As same as the
orthogonal-symplectic case, the (strong) Howe duality for
$\wt{\Omega}$ is equivalent to the theta dichotomy for unitary
isometry dual pairs. We explain this phenomenon briefly. Up to
isometry, there are two different hermitian spaces over $E$ of
dimension $m\geq1$: $V^\pm$ defined by
$$\epsilon(V^\pm)=\epsilon_{E/F}((-1)^{\frac{m(m-1)}{2}}\det
V^\pm)=\pm1.$$ To avoid confusion, we write $\CR(H_1,V^\pm)$ to
emphasize its dependence on $V^\pm$. Applying the same arguments of
\cite[Lemma 6.1]{ro}, we have the following principle for the
(strong) Howe duality for $\wt{\Omega}$.

\begin{prop}
Assume $p>2$. Then the (strong) Howe duality holds for $\wt{\Omega}$
if and only if $\CR(H_1,V^+)\cap\CR(H_1,V^-)=\emptyset$.
\end{prop}

\begin{remark}
It is conjectured in \cite{hks} that for $m\leq n$,
$$\CR(H_1,V^+)\cap\CR(H_1,V^-)=\emptyset,$$ which is called theta
dichotomy. Theta dichotomy has been proved (cf. \cite[Theorem
2.1.7]{ha2}) under the assumption of the so called ``weak
conservation relation''. Recently, Sun and Zhu \cite{sz} proved the
so called ``conservation relation'' for all type I irreducible dual
pairs and all local fields of characteristic zero. In summary, theta
dichotomy has been proved. It is also known that for $m\geq 2n$,
$$\CR(H_1,V^+)\cap\CR(H_1,V^-)\neq\emptyset.$$ Thus the Howe duality
does not hold for $\wt{\Omega}$ when $m\geq 2n$.
\end{remark}

\paragraph{Acknowledgements} This paper is part of the author's thesis in
2011. He expresses his sincere gratitude to Professor Linsheng Yin
and Professor Ye Tian for their constant encouragement and support.
The author would like to thank Professor Wee Teck Gan for his
comments and mention of \cite[Proposition 3.2]{gt}, and would like
to thank the referees for their careful readings of the manuscript
and many comments leading to a much improved exposition.

\s{Chong Zhang\\
Academy of Mathematics and System Science, Chinese Academy of
Sciences,\\
Beijing, 100190, PR China.\\
E-mail address: \texttt{zhangchong02@gmail.com}}


\begin{thebibliography}{XXXX}
\addtocontents{Bibliography}

\bibitem[Ba]{ba} L. Barthel, {\em Local Howe correspondence for groups of
similitudes}, J. reine angew. Math. {\bf 414}, 207-220 (1991).

\bibitem[GT]{gt} W. T. Gan, W. Tantono, {\em The local Langlands conjecture for $\GSp(4)$ II:
the case of inner forms}, preprint.

\bibitem[Ha1]{ha1} M. Harris, {\em L-functions of $2\times2$ unitary groups
and factorization of periods of Hilbert modular forms}, J. Amer.
Math. Soc. {\bf 6}, 637-719 (1993).

\bibitem[Ha2]{ha2} M. Harris, {\em Cohomological automorphic forms on unitary groups, II:
period relations and values of L-functions}, in ``Harmonic analysis,
Group Representations, Automorphic Forms and Invariant Theory: In
Honor of Roger E. Howe", vol. {\bf12}, 89-149, Lecture Notes Series,
Institue for Mathematical Sciences, National University of
Singapore, World Scientific Publishing, 2007.


\bibitem[HKS]{hks} M. Harris, S. Kudla, S. Sweet, {\em Theta dichotomy for
unitary groups}, J. Amer. Math. Soc. {\bf 9}, 941-1004 (1996).

\bibitem[Ku]{ku} S. Kudla, {\em Splitting metaplectic covers of dual
reductive pairs}, Israel J. Math. {\bf 87}, 361-401 (1994).

\bibitem[MVW]{mvw} C. Moeglin, M.-F. Vig\'{e}nras, J.-L. Waldspurger, {\em
Correspondances de Howe sur un corps $p$-adique}, Lecture Notes in
Math. 1291, Springer-Verlag, Berlin-Heidelberg-New York (1987).

\bibitem[Ra]{ra} R. R. Rao, {\em On some explicit formulas in the theory of Weil representation}, Pacific J. Math. {\bf 157}, 335-371 (1993).

\bibitem[Ro]{ro} B. Roberts, {\em The theta correspondence for
similitudes}, Israel J. Math. {\bf 94}, 285-317 (1996).

\bibitem[SZ]{sz} B. Sun, C.-B. Zhu, {\em Conservation relations for local theta
correspondence}, arXiv: 1204.2969.

\bibitem[Wa]{wa} J.-L. Waldspurger, {\em D\'{e}monstration
d'une conjecture de dualit\'{e} de Howe dans le cas
$p$-adique, $p\neq2$}, in: Festschrift in honor of
Piatetski-Shapiro, Israel Math. Conf. Proc., vol. {\bf 2}, 267-324
(1990).

\end{thebibliography}
\end{document}